\documentclass[amsfonts]{amsart}

\usepackage{amsmath,amssymb,amsthm,eucal}
\usepackage[dvipdfmx, colorlinks=true, backref=page]{hyperref}
\usepackage[all]{xy}

\numberwithin{equation}{section}

\SelectTips{eu}{12}

\newtheorem{theorem}{Theorem}[section]
\newtheorem{proposition}[theorem]{Proposition}
\newtheorem{lemma}[theorem]{Lemma}
\newtheorem{corollary}[theorem]{Corollary}

\theoremstyle{definition}

\newtheorem{problem}[theorem]{Problem}

\theoremstyle{remark}

\newcommand{\Z}{\mathbb{Z}}

\newcommand{\R}{\mathbb{R}}

\newcommand{\h}{\mathtt{h}}

\title{Van Kampen-Flores theorem and Stiefel-Whitney classes}

\author{Daisuke Kishimoto}
\address{Faculty of Mathematics, Kyushu University, Fukuoka 819-0395, Japan}
\email{kishimoto@math.kyushu-u.ac.jp}

\author{Takahiro Matsushita}
\address{Department of Mathematical Sciences, University of the Ryukyus Nishihara-cho, Okinawa 903-0213, Japan}
\email{mtst@sci.u-ryukyu.ac.jp}



\begin{document}

  \maketitle

  %%%%% Abstract %%%%%

  \begin{abstract}
    The van Kampen-Flores theorem states that the $d$-skeleton of a $(2d+2)$-simplex does not embed into $\R^{2d}$. We prove the van Kampen-Flores theorem for triangulations of manifolds satisfying a certain condition on their Stiefel-Whitney classes. In particular, we show that the $d$-skeleton of a triangulation of a $(2d+1)$-manifold with non-trivial total Stiefel-Whitney class does not embed into $\R^{2d}$.
  \end{abstract}

  %%%%% Section 1 %%%%%

  \section{Introduction}

  The van Kampen-Flores theorem states that the $d$-skeleton of a $(2d+2)$-simplex does not embed into $\R^{2d}$. Actually, van Kampen \cite{K} and Flores \cite{F} proved the following stronger statement which is now called the van Kampen-Flores theorem too. For any continuous map $f\colon \Delta^{2d+2}_d\to\R^{2d}$, there are disjoint simplices $\sigma,\tau$ of $\Delta^{2d+2}_d$ such that $f(\sigma)$ and $f(\tau)$ have a point in common, where $\Delta^n$ and $K_d$ denote the $n$-simplex and the $d$-skeleton of a simplicial complex $K$, respectively. See \cite{BFZ} for a sharpened version. There are several generalizations of the van Kampen-Flores theorem \cite{BFZ,BZ,GMPPTW,JPZ,JVZ,KM,MSP,NW,S,V}. Among others, the authors
%Kishimoto and Matsushita
\cite{KM} gave a generalization to maps out of certain CW complexes into the Euclidean space, which includes the following case. Let $K$ be a triangulation of a mod 2 homology $(2d+1)$-sphere. Then for any continuous map $f\colon K_d\to\R^{2d}$, there are disjoint simplices $\sigma,\tau$ of $K_d$ such that $f(\sigma)$ and $f(\tau)$ have a point in common. (cf. \cite{NW})

  We also have an analogous non-embeddability result for the 1-skeleton of a triangulation of a surface as follows. Let $G$ be the 1-skeleton of a triangulation of a closed surface $S$. Then we have $e=3v-3\chi(S)$, where $v$ and $e$ are the numbers of vertices and edges of $G$ and $\chi(S)$ denotes the Euler characteristic of $S$. On the other hand, as in \cite[Corollary 9.5.1]{BM}, we have $e\le 3v-6$ whenever $G$ is planar, and so the statement is proved. Motivated by the above two results on the non-embeddability of skeleta of manifold triangulations, we pose the following problem, where manifolds will be assumed to be paracompact, smooth and without boundary throughout the paper.

  \begin{problem}
    Find a class of manifolds such that the $d$-skeleta of their triangulations are not embeddable into $\R^{2d}$.
  \end{problem}

  In this paper, we give a class of manifolds for which the $d$-skeleta of their triangulations are not embeddable into $\R^{2d}$. We will take an approach different from the previous works \cite{HKTT,KM} in which we assumed high acyclicity. Instead of acyclicity, we consider a condition on characteristic classes of manifolds. To state the main result, we set notation. For $0\le k\le d-1$, we define polynomials $a_1^{(k)},\ldots,a_d^{(k)}\in\Z_2[x_1,x_2,\ldots,x_d]$ as follows. Let
  \[
    a_i^{(0)}=x_i\quad(1\le i\le d)
  \]
  and
  \[
    a_i^{(k)}=x_ia_1^{(k-1)}+a_{i+1}^{(k-1)}\quad(1\le i\le d)
  \]
  for $1\le k\le d-1$, where we set $a_{d+1}^{(k-1)}=0$. For a $d$-manifold $M$, we define
  \[
    w^{(k)}_i(M)=a_i^{(k)}(w_1(TM),\ldots,w_d(TM))
  \]
  for $1\le i\le d$, and set $w^{(k)}(M)=1+w_1^{(k)}(M)+\cdots+w_d^{(k)}(M)$, where $TM$ denotes the tangent bundle of $M$. Note that $w^{(0)}(M)$ coincides with the total Stiefel-Whitney class $w(M)$. Now we state the main theorem.

  \begin{theorem}
    \label{main}
    Let $K$ be a triangulation of a $(2d+k+1)$-manifold $M$ with $w^{(k)}(M)\ne 1$, where $0\le k\le d-1$. Then for any continuous map $f\colon K_{d+k}\to\R^{2d+2k}$, there are disjoint simplices $\sigma,\tau$ of $K_{d+k}$ such that $f(\sigma)$ and $f(\tau)$ have a point in common.
  \end{theorem}

  The $k=0$ case of Theorem \ref{main} is of particular interest.

  \begin{corollary}
    \label{k=0}
    Let $K$ be a triangulation of a $(2d+1)$-manifold $M$ with $w(M)\ne 1$. Then for any continuous map $f\colon K_d\to\R^{2d}$, there are disjoint simplices $\sigma,\tau$ of $K_d$ such that $f(\sigma)$ and $f(\tau)$ have a point in common.
  \end{corollary}

  Let $M$ be a $d$-manifold. In general, it is not easy to check the non-triviality of $w^{(k)}(M)$. However, the non-triviality of $w^{(1)}(M)$ can be deduced from standard topological conditions. By definition, we have
  \[
    w_d^{(1)}(M)=w_1(TM)w_{d-1}(TM)+w_d(TM).
  \]
  If $M$ is orientable, then $w_1(M)=0$. Suppose $M$ is closed and connected. Then by the Euler-Poincar\'e theorem, $w_d(M)$ coincides with the mod 2 reduction of the Euler characteristic $\chi(M)$ in $H^d(M;\Z_2)\cong\Z_2$. So $w_d^{(1)}(M)\ne 0$ whenever $M$ is orientable and $\chi(M)$ is odd, and thus by Theorem \ref{main}, we obtain the following corollary which can be thought of as a generalization of the above mentioned non-embeddability of the 1-skeleton of a triangulation of a closed surface into a plane.

  \begin{corollary}
    Let $K$ be a triangulation of a closed connected orientable $(2d+2)$-manifold. If $\chi(K)$ is odd, then for any continuous map $f\colon K_{d+1}\to\R^{2d+2}$, there are disjoint simplices $\sigma,\tau$ of $K_{d+1}$ such that $f(\sigma)$ and $f(\tau)$ have a point in common.
  \end{corollary}

The total Stiefel-Whitney class of a mod 2 homology sphere is trivial, and so we cannot apply Corollary \ref{k=0} to it. But as in \cite{HKK}, the van Kampen-Flores theorem holds for a triangulation of a mod 2 homology sphere. Besides mod 2 homology spheres, there are important classes of manifolds with trivial total Stiefel-Whitney classes such as (real) moment-angle manifolds as in \cite{HKK}. So it would be interesting to find whether or not the van Kampen-Flores theorem holds for such manifolds.

  In Section 2, we recall the Stiefel-Whitney height of a free $\Z_2$-space, and compute it for the deleted product of a manifold with $w^{(k)}(M)\ne 1$. In Section 3, we apply the result of Section 2 to prove the topological Radon theorem for manifold triangulations (Theorem \ref{Radon}), and then prove Theorem \ref{main} by the constraint method \cite{BFZ} applied to it.

\subsection*{Acknowledgements}
The first author was partly supported by JSPS KAKENHI Grant Numbers JP17K05248 and JP19K03473, and the second author was partly supported by JSPS KAKENHI Grant Numbers JP19K14536 and JP23K12975.
The authors would like to thank Kenta Ozeki for useful information on embeddings of graphs.

  %%%%% Section 2 %%%%%

  \section{Stiefel-Whitney height}

  In this section, we recall the Stiefel-Whitney height of a free $\Z_2$-space, and compute it for the deleted product of a manifold by comparing the deleted product with the tangent sphere bundle, similarly to the tangent microbundle \cite{M1}.

  We define the Stiefel-Whitney class of a free $\Z_2$-space as in \cite[Definition 8.21]{K}. Let $X$ be a free $\Z_2$-space. If $X$ has the $\Z_2$-homotopy type of a $\Z_2$-complex, then there is a $\Z_2$-map $X\to S^\infty$ which is unique, up to $\Z_2$-homotopy. See \cite[Proposition 8.16]{K}, for example. In particular, this map induces a map $\rho \colon X/\Z_2 \to S^\infty/ \Z_2 = \R P^\infty$, which is unique, up to homotopy. The map $\rho$ is often called the classifying map of $X$. The Stiefel-Whitney class of $X$ is defined by
  \[
    \varpi_1(X)=\rho^*(a)\in H^1(X/\Z_2;\Z_2)
  \]
  where $a$ is a generator of $H^1(\R P^\infty;\Z_2)\cong\Z_2$. By the uniqueness of the classifying map, the Stiefel-Whitney class of a free $\Z_2$-space satisfies the naturality such that for a $\Z_2$-map $f\colon X\to Y$ between free $\Z_2$-spaces,
  \begin{equation}
    \label{naturality}
    \bar{f}^*(\varpi_1(Y))=\varpi_1(X)
  \end{equation}
  where the map $\bar{f}\colon X/\Z_2\to Y/\Z_2$ is covered by $f$.

  Now we define the Stiefel-Whitney height of $X$ by
  \[
    \h(X)=\max\{n\ge 0\mid\varpi_1(X)^n\ne 0\}.
  \]
  Clearly, the Stiefel-Whitney height is an invariant of the $\Z_2$-homotopy type of a free $\Z_2$-space. Analogous invariants for more general group actions are considered in \cite{FH,HK,HKTT,KM,V}. We record basic properties of the Stiefel-Whitney height that we are going to use, where we will freely use the fact that each skeleton of a $\Z_2$-complex is a $\Z_2$-complex.

  \begin{lemma}
    \label{h(X)}
    Let $X,Y$ be free $\Z_2$-complexes.
    \begin{enumerate}
      \item There is an inequality
      \[
        \h(X)\le\dim X.
      \]

      \item If there is a $\Z_2$-map $f\colon X\to Y$, then
      \[
        \h(X)\le\h(Y).
      \]

      \item If $X$ is a free $\Z_2$-complex and $\h(X)\ge n$, then
      \[
        \h(X_n)=n.
      \]
    \end{enumerate}
  \end{lemma}

  \begin{proof}
    (1) Since $X/\Z_2$ is a CW complex of dimension $n$, it follows from the cellular approximation theorem that the classifying map $X/\Z_2\to\R P^\infty$ can be compressed into $\R P^n\subset\R P^\infty$, where $n=\dim X$. Then we get $\varpi_1(X)^{n+1}=0$, implying $\h(X)\le n$.

    \noindent(2) Let $\bar{f}\colon X/\Z_2\to Y/\Z_2$ be the map covered by $f$. Then by \eqref{naturality}, we have $\bar{f}^*(\varpi_1(Y))=\varpi_1(X)$, and so
    \[
      0\ne \varpi_1(X)^n=\bar{f}^*(\varpi_1(Y))^n=\bar{f}^*(\varpi_1(Y)^n)
    \]
    whenever $\varpi_1(X)^n\ne 0$. Thus $\varpi_1(Y)^n\ne 0$, implying $\h(X)\le\h(Y)$.

    \noindent(3) Since $X$ is a free $\Z_2$-complex, $X_n$ is also a free $\Z_2$-complex, and so we can consider $\h(X_n)$. Let $j\colon X_n/\Z_2\to X/\Z_2$ denote the inclusion. Then since $j^*(\varpi_1(X))=\varpi_1(X_n)$ by \eqref{naturality}, we get
    \[
      \varpi_1(X_n)^n=j^*(\varpi_1(X))^n=j^*(\varpi_1(X)^n)\ne 0.
    \]
    Since the induced map $j^*\colon H^n(X/\Z_2;\Z_2)\to H^n(X_n/\Z_2;\Z_2)$ is injective, we obtain $\varpi_1(X_n)^n\ne 0$, implying $\h(X)\ge n$. On the other hand, we have $\h(X_n)\le n$ by (1), and thus we obtain $\h(X_n)=n$, as stated.
  \end{proof}

  We compute the Stiefel-Whitney height of the deleted product of a manifold. The deleted product of a topological space $X$ is defined by
  \[
    X^\star=\{(x,y)\in X\times X\mid x\ne y\}.
  \]
  Then $X^\star$ is a free $\Z_2$-space by the involution
  \[
    X^\star\to X^\star,\quad(x,y)\mapsto(y,x).
  \]
  Let $M$ be a manifold of dimension $d$. We equip $M$ with a Riemannian metric. For each $x\in M$, we define $\delta_x>0$ so as to satisfy that for some open neighborhood $U_x$ of $x$ in $M$, the exponential map $\mathrm{exp}\colon TM\to M$ is injective on the subsets $\{v\in T_yM\mid|v|\le\delta_x\}$ for each $y\in U_x$ (see \cite[Lemma 10.3]{M}). Let $\mathcal{U}$ denote a locally finite refinement of an open cover $M=\bigcup_{x\in M}U_x$, which exists because $M$ is paracompact. Let $\{\rho_U\}_{U\in\mathcal{U}}$ be a partition of unity subordinate to $\mathcal{U}$. We choose one point $x_U\in U$ for each $U\in\mathcal{U}$, and let
  \[
    \delta(x)=\sum_{U\in\mathcal{U}}\rho_U(x)\delta_{x_U}
  \]
  which is a continuous function on $M$. Now we define
  \[
    S=\{v\in TM\mid |v|=\delta(\pi(v))\}
  \]
  where $\pi\colon TM\to M$ denotes the projection. Then $S$ is a tangent sphere bundle on $M$. Consider a map
  \begin{equation}
    \label{g}
    S\to M\times M,\quad v\mapsto(\mathrm{exp}(v),\mathrm{exp}(-v)).
  \end{equation}
  For each $x\in M$, there is $U\in\mathcal{U}$ such that $x\in U$ and $\delta(x)\le\delta_U$. Then the exponential map $\mathrm{exp}\colon TM\to M$ is injective on each $S_x$, the fiber of the tangent sphere bundle $S\to M$ at $x\in M$. So image of the map \eqref{g} is included in $M^\star$, and so we get a map $g\colon S\to M^\star$. Let $\Z_2$ act on $S$ by the involution
  \[
    S\to S,\quad v\mapsto-v.
  \]
  Then the map $g$ is a $\Z_2$-map, and so by Lemma \ref{h(X)}, we get
  \begin{equation}
    \label{h(S)}
    \h(S)\le\h(M^\star).
  \end{equation}
  By \eqref{naturality}, $\varpi_1(S)$ restricts to $\varpi_1(S_x)$ for each $x\in M$. Since $S_x$ is identified with $S^{d-1}$ with antipodal $\Z_2$-action, $\varpi_1(S_x)$ is a ganerator of $H^1(S_x/\Z_2;\Z_2)\cong\Z_2$, where $S_x/\Z_2=\R P^{d-1}$. Then since $S/\Z_2$ is identified with the projectivization of $TM$, it follows from by \cite[(19.2.2) Theorem and Section 19.4]{tD} that there is an isomorphism
  \[
    H^*(S/\Z_2;\Z_2)\cong H^*(M;\Z_2)\{1,\varpi_1(S),\varpi_1(S)^2\ldots,\varpi_1(S)^{d-1}\}
  \]
  as an $H^*(M;\Z_2)$-module such that the equality
  \begin{equation}
    \label{w(S)}
    \varpi_1(S)^d+w_1(TM)\varpi_1(S)^{d-1}+w_2(TM)\varpi_1(S)^{d-2}+\cdots+w_d(TM)=0
  \end{equation}
  holds, where $R\{x_1,\ldots,x_n\}$ denotes the free $R$-module with basis $x_1,\ldots,x_n$ for a commutative ring $R$.

  \begin{proposition}
    \label{h(M)}
    If $M$ is a $d$-manifold with $w^{(k)}(M)\ne 1$, then
    \[
      \h(M^\star)\ge d+k.
    \]
  \end{proposition}

  \begin{proof}
    By \eqref{w(S)} and the definition of $w^{(k)}_i(M)$, we have
    \[
      \varpi_1(S)^{d+k}+w_1^{(k)}(M)\varpi_1(S)^{d-1}+w_2^{(k)}(M)\varpi_1(S)^{d-2}+\cdots+w_d^{(k)}(M)=0.
    \]
    Then we get $\varpi_1(S)^{d+k}\ne 0$, implying $\h(S)\ge d+k$, whenever $w^{(k)}(M)\ne 1$. Thus by \eqref{h(S)}, we obtain $\h(M^\star)\ge d+k$, as stated.
  \end{proof}

  %%%%% Section 3 %%%%%

  \section{Topological Radon theorem}

  The topological Radon theorem \cite{BB} states that for any continuous map $f\colon\Delta^{d+1}\to\R^d$, there are disjoint simplices $\sigma,\tau$ of $\Delta^{d+1}$ such that $f(\sigma)$ and $f(\tau)$ have a point in common. In this section, we prove the  topological Radon theorem for a map out of a manifold triangulation (Theorem \ref{Radon}), and then prove Theorems \ref{main} by the constraint method \cite{BFZ} applied to it.

  For a simplicial complex $K$, we define the simplicial deleted product $\widehat{K}^\star$ as the subcomplex of a CW complex $K\times K$ consisting of closed cells $\sigma\times\tau$ where $\sigma,\tau$ are disjoint simplices of $K$. Then $\widehat{K}^\star$ is a subspace of $K^\star$ such that the $\Z_2$-action on $K^\star$ restricts to $\widehat{K}^\star$. We connect the existence of a certain map out of $K$ to a $\Z_2$-map out of a skeleton of $\widehat{K}^\star$.

  \begin{lemma}
    \label{embedding}
    Let $K$ be a simplicial complex. If there is a map $f\colon K\to\R^d$ such that $f(\sigma)\cap f(\tau)=\emptyset$ whenever $\sigma,\tau$ are disjoint simplices of $K$ with $\dim\sigma+\dim\tau\le n$, then there is a $\Z_2$-map
    \[
      (\widehat{K}^\star)_n\to(\R^d)^\star.
    \]
  \end{lemma}

  \begin{proof}
    The $\Z_2$-map
    \[
      K\times K\to\R^d\times\R^d,\quad(x,y)\mapsto (f(x),f(y))
    \]
    restricts to a $\Z_2$-map $(\widehat{K}^\star)_n\to(\R^d)^\star$, where the $\Z_2$-action is given by switching the direct product factors. Indeed, every closed cell of $(\widehat{K}^\star)_n$ is of the form $\sigma\times\tau$ for disjoint simplices $\sigma,\tau$ of $K$ with $\dim\sigma+\dim\tau\le n$, and for such simplices of $K$, we have $f(\sigma)\cap f(\tau)=\emptyset$, that is, $(f\times f)(\sigma\times\tau)=f(\sigma)\times f(\tau)\subset(\R^d)^\star$.
  \end{proof}

  We prove a version of the Borsuk-Ulam theorem.

  \begin{proposition}
    \label{Borsuk-Ulam}
    Let $K$ be a simplicial complex. If $\h(\widehat{K}^\star)\ge d$, then for any continuous map $f\colon K\to\R^d$, there are disjoint simplices $\sigma,\tau$ of $K$ with $\dim\sigma+\dim\tau\le d$ such that $f(\sigma)$ and $f(\tau)$ have a point in common.
  \end{proposition}

  \begin{proof}
    By Lemma \ref{h(X)}, we have $\h(\widehat{K}_d)=d$. Suppose there is a map $f\colon K\to\R^d$ such that for disjoint simplices $\sigma,\tau$ of $K$ with $\dim\sigma+\dim\tau\le d$, $f(\sigma)\cap f(\tau)=\emptyset$. Then by Lemma \ref{embedding}, there is a $\Z_2$-map $(\widehat{K}^\star)_d\to(\R^d)^\star$, and so by Lemma \ref{h(X)}, we get
    \[
      d=\h((\widehat{K}^\star)_d)\le\h((\R^d)^\star)
    \]
    On the other hand, $(\R^d)^\star$ is $\Z_2$-homotopy equivalent to the unit sphere of the orthogonal complement of the diagonal subspace in $\R^d\times\R^d$, which is a $(d-1)$-sphere. Then $\h((\R^d)^\star)=d-1$, and we obtain a contradiction. Thus the proof is finished.
  \end{proof}

  In order to apply Proposition \ref{h(M)} to Proposition \ref{Borsuk-Ulam}, we need the following lemma. The lemma is stated in \cite[Lemma 2.1]{Sh}, but the proof is false: the map $\beta$ is not continuous. So we give a fix.

  \begin{lemma}
    \label{deleted product}
    For a simplicial complex $K$, the inclusion $\widehat{K}^\star\to K^\star$ is a $\Z_2$-homotopy equivalence.
  \end{lemma}

  \begin{proof}
    Since both $\widehat{K}^\star$ and $K^\star$ are free $\Z_2$-spaces and the inclusion $\widehat{K}^\star \to K^\star$ is a $\Z_2$-map, it is sufficient to show that the inclusion is a homotopy equivalence. So we construct a deformation retract of $K^\star$ onto $\widehat{K}^\star$. We may assume that vertices of $K$ are $(1,0,\ldots,0),(0,1,0,\ldots,0),\ldots,(0,\ldots,0,1)\in\R^m$. Then every point of $K$ is given by $(t_1,\ldots,t_m)\in\R^m$ for some $t_1,\ldots,t_m\ge 0$ satisfying $t_1+\cdots+t_m=1$. For $x=(x_1,\ldots,x_m),y=(y_1,\ldots,y_m)\in\R^m$ and $i=1,2,\ldots,m$, let
    \[
      \delta_i(x,y)=\max\{x_i-y_i,0\}
    \]
    Then $\delta_i(x,y)>0$ for some $i$ whenever $(x,y)\in K^\star$, and so we can define a continuous map
    \[
      \alpha\colon K^\star\to K,\quad(x,y)\mapsto\frac{1}{\delta_1(x,y)+\cdots+\delta_m(x,y)}(\delta_1(x,y),\ldots,\delta_m(x,y)).
    \]
    Now we define a map
    \[
      K^\star\times[0,1]\to K^\star,\quad((x,y),t)\mapsto((1-t)x+t\alpha(x,y),(1-t)y+t\alpha(y,x)).
    \]
    It is straightforward to check that this is a deformation retraction of $K^\star$ onto $\widehat{K}^\star$, completing the proof.
  \end{proof}

  We are ready to prove the topological Radon theorem for maps out of manifold triangulations.

  \begin{theorem}
    \label{Radon}
    Let $K$ be a triangulation of a $d$-manifold with $w^{(k)}(M)\ne 1$, where $1\le k\le d-1$. Then for any continuous map $f\colon K\to\R^{d+k}$, there are disjoint simplices $\sigma,\tau$ of $K$ with $\dim\sigma+\dim\tau\le d+k$ such that $f(\sigma)$ and $f(\tau)$ have a point in common.
  \end{theorem}

  \begin{proof}
    By Proposition \ref{h(M)} and Lemma \ref{deleted product}, we have $\h(\widehat{K}^\star)\ge d+k$. Then by Proposition \ref{Borsuk-Ulam}, the statement is proved.
  \end{proof}

  Finally, we prove Theorem \ref{main} by the constraint method \cite{BFZ} applied to Theorem \ref{Radon}.

  \begin{proof}
    [Proof of Theorem \ref{main}]
    Suppose there is a map $f\colon K_{d+k}\to\R^{2d+2k}$ such that for any distinct simplices $\sigma,\tau$ of $K_{d+k}$, $f(\sigma)$ and $f(\tau)$ do not have a point in common. Then since $(K,K_{d+k})$ is an NDR pair, we can extend $f$ to a map $\tilde{f}\colon K\to\R^{2d+{2k}}$. On the other hand, there is a map $c\colon K\to\R$ such that $c^{-1}(0)=K_{d+k}$ because $(K,K_{d+k})$ is an NDR pair. We define a map
    \[
      F\colon K\to\R^{2d+2k+1},\quad x\mapsto(\tilde{f}(x),c(x)).
    \]
    By Theorem \ref{Radon}, there are disjoint faces $\sigma,\tau$ of $K$ with $\dim\sigma+\dim\tau\le 2d+2k+1$ such that for some $x\in\sigma,y\in\tau$, we have $F(x)=F(y)$, that is, $\tilde{f}(x)=\tilde{f}(y)$ and $c(x)=c(y)$. By taking smaller simplices if necessary, we may assume that $x\in\mathrm{Int}(\sigma)$ and $y\in\mathrm{Int}(\tau)$, where $\mathrm{Int}(\nu)$ denotes the interior of a cell $\nu$. If $c(x)\ne 0$, then $c(y)\ne 0$. So since $x\in \mathrm{Int}(\sigma)$ and $y\in\mathrm{Int}(\tau)$, we get $0\notin c(\mathrm{Int}(\sigma))$ and $0\notin c(\mathrm{Int(\tau}))$, implying $\dim\sigma\ge d+k+1$ and $\dim\tau\ge d+k+1$. Then $\dim\sigma+\dim\tau>2d+2k+1$, which is a contradiction, hence $c(x)=0$. In this case, we have $c(y)=0$. Then since $x\in\mathrm{Int}(\sigma)$ and $y\in\mathrm{Int}(\tau)$,  we get that $\sigma,\tau$ are simplices of $K_{d+k}$ such that $f(\sigma)\cap f(\tau)\supset\{f(x)=f(y)\}\ne\emptyset$. This is a contradiction too. Thus our supposition, the existence of the map $f$, is false, completing the proof.
  \end{proof}


\begin{thebibliography}{99}
    \bibitem{BB} E.G. Bajm\'oczy, I. B\'ar\'any, A common generalization of Borsuk's and Radon's theorem, Acta. Math. Hungarica \textbf{34} (1979), 347-350.

    \bibitem{BFZ} P.V.M. Blagojevi\'{c}, F. Frick, and G.M. Ziegler, Tverberg plus constraints, Bull. Lond. Math. Soc. \textbf{46}, (2014), 953-967.

    \bibitem{BZ} P.V.M. Blagojevi\'{c} and G.M. Ziegler, Beyond the Borsuk-Ulam theorem: The topological Tverberg story, A journey through discrete mathematics, 273-341, Springer, Cham, 2017.

    \bibitem{BM} J.A. Bondy and U.S.R. Murty, Graph Theory with Applications, American Elsevier Publishing Co., Inc., New York, 1976.

    %\bibitem{DII} Z. Dzedzej, A. Idzik, and M. Izydorek, Borsuk-Ulam type theorems on product spaces II, Topol. Methods Nonlinear Anal. \textbf{14} (1999), 345-352.

    \bibitem{FH} E. Fadell and S. Husseini, An ideal-valued cohomological index theory with applications to Borsuk-Ulam and Bourgin-Yang theorems, Ergodic Theory Dynam. Systems, \textbf{8} (1988), 73-85.

    %\bibitem{Flapan} E. Flapan, When topology meets chemistry: a topological look at molecular chirality, Cambridge University Press, 2000.

    \bibitem{F} A. Flores, \"Uber $n$-dimensionale Komplexe die im $R_{2n+1}$ absolut selbstverschlungen sind, Ergeb. Math. Kolloq. \textbf{6} (1932/1934), 4-7.

    %\bibitem{FH} F. Frick and M. Harrison, Spaces of embeddings: Nonsingular bilinear maps, chirality, and their generalizations, Proc. Amer. Math. Soc. \textbf{150} (2022), no. 1, 423-437.

    \bibitem{GMPPTW} X. Goaoc, I. Mabillard, P. Pat\'{a}k, Z. Pat\'{a}kov\'{a}, M. Tancer, and U. Wagner, On generalized Heawood inequalities for manifolds: a van Kampen-Flores-type nonembeddability result. Israel J. Math. \textbf{222} (2017), no. 2, 841-866.

    \bibitem{HK} Y. Hara and D. Kishimoto, Note on the cohomology of finite cyclic coverings, Topology Appl. \textbf{160} (2013), no. 9, 1061-1065.

    \bibitem{HKK} S. Hasui, D. Kishimoto, and A. Kizu, The Stiefel-Whitney classes of moment-angle manifolds are trivial, Forum Math. \textbf{34} (2022), no. 6, 1463-1474.

    \bibitem{HKTT} S. Hasui, D. Kishimoto, M. Takeda, and M. Tsutaya, Tverberg's theorem for cell complexes, Bull. London. Math. Soc., Volume 55, Issue 4, (2023) 1944-1956.

    %\bibitem{JS} S. Jackowski and J. S\l{}omi\'{n}ska, $G$-functors, $G$-posets and homotopy decompositions of $G$-spaces, Fund. Math. \textbf{169} (2001), 249-287.

    \bibitem{JPZ} D. Joji\'{c}, G. Panina, R. \u{Z}ivaljevi\'{c}, A Tverberg type theorem for collectively unavoidable complexes, Israel J. Math. \textbf{241} (2021), no. 1, 17-36.

    \bibitem{JVZ} D. Joji\'{c}, S.T. Vre\'{c}ica, and R.T. \u{Z}ivaljevi\'{c}, Symmetric multiple chessboard complexes and a new theorem of Tverberg type, J. Alg. Comb \textbf{46} (2017), 15-31.

    \bibitem{K} E.R. van Kampen, Komplexe in euklidischen R\"aumen, Abh. Math. Seminar Univ. Hamburg \textbf{9} (1933), 72-78.

    %\bibitem{KL} D. Kishimoto and R. Levi, Polyhedral products over finite posets, accepted by Kyoto J. Math.

    \bibitem{KM} D. Kishimoto and T. Matsushita, van Kampen-Flores theorem for cell complexes, accepted by Discrete Comput. Geom.

    \bibitem{K}  D.N. Kozlov, Combinatorial Algebraic Topology, Algorithms and Computation in Mathematics \textbf{21}, Springer-Verlag Berlin Heidelberg, 2008.

    %\bibitem{M} J. Matou\v{s}ek, Using the Borsuk-Ulam theorem, Lectures on topological methods in combinatorics and geometry, written in cooperation with A. Bj\"{o}rner and G.M. Ziegler, Universitext, Springer-Verlag, Berlin, 2003.

    %\bibitem{R} Y.B. Rudyak, On category weight and its applications, Topology \textbf{38} (1999), no. 1, 37-55.

    \bibitem{MSP} L. Mart\'{i}nez-Sandoval and A. Padrol, The convex dimension of hypergraphs and the hypersimplicial Van Kampen-Flores Theorem, J. Comb. Theory B \textbf{149} (2021), 23-51.

    \bibitem{M} J.W. Milnor, Morse Theory, Annals of Mathematics Studies \textbf{51}, Princeton University Press, Princeton, N.J., 1963.

    \bibitem{M1} J. Milnor. Microbundles. I, Topology \textbf{3} (1964), No. suppl, suppl. 1, 53-80.

    \bibitem{NW} E. Nevo and U. Wagner, On the embeddability of skeleta of spheres, Israel J. Math. \textbf{174} (2009), 381-402.

    \bibitem{S} K.S. Sarkaria, A generalized van Kampen-Flores theorem, Proc. Amer. Math. Soc. \textbf{111} (1991), no. 2, 559-565.

    \bibitem{Sh}A. Shapiro, Obstructions to the imbedding of a complex in a Euclidean space. I. The first obstruction, Ann. of Math. \textbf{99} (1957), No. 2, 256-269.

    \bibitem{tD} T. tom Dieck, Algebraic topology, EMS Textbooks in Mathematics, European Mathematical Society (EMS), Z\"{u}rich, 2008.

    \bibitem{V} A.Y. Volovikov, On the van Kampen-Flores theorem, Math. Notes \textbf{59} (1996), no. 5, 477-481.

    %\bibitem{ZZ} G.M. Ziegler and R.T. \u{Z}ivaljevi\'{c}, Homotopy types of subspace arrangements via diagrams of spaces, Math. Ann. \textbf{295} (1993), 527-548.
  \end{thebibliography}
\end{document}